\documentclass[12pt]{article}
\usepackage[english]{babel}

\usepackage[utf8]{inputenc}

\usepackage{amsmath,amsthm,amsfonts,amssymb,bbm}
\usepackage{mathrsfs}
\usepackage[]{natbib} 
\usepackage[pagebackref=true,linktoc=page,colorlinks=true,linkcolor=blue,citecolor=blue]{hyperref}
\usepackage{hyperref}
\usepackage{enumerate}

\usepackage{enumitem}
\usepackage{makecell}
\usepackage{multirow}
\usepackage{graphicx}
\usepackage{here}
\usepackage[all]{xy}
\usepackage{tabularx} 
\usepackage{booktabs} 
\usepackage{accents} 
\usepackage[font=small,labelfont=bf]{caption} 
\usepackage{dsfont}
\usepackage{pstool,psfrag}

\usepackage{accents}

\newtheorem{Lemma}{Lemma}
\newtheorem{Theorem}[Lemma]{Theorem}

\newtheorem{Proposition}[Lemma]{Proposition}
\newtheorem{Remark}[Lemma]{Remark}

\newtheorem{Example}{Example}

\newcommand{\M}{\mathbb{M}}
\newcommand{\R}{\mathbb{R}}

\newcommand{\N}{\mathbb{N}}
\newcommand{\X}{\mathbb{X}}
\newcommand{\limm}{\underset{n \rightarrow \infty}{\lim}}

\renewcommand{\P}{\mathbb{P}}

\newcommand{\E}{\mathbb{E}}
\renewcommand{\1}{\mathbbm{1}}

\renewcommand{\d}{\mathrm{d}}

\newcommand{\xxi}{\boldsymbol{\xi}}
\newcommand{\mxi}{m_{\xi}}
\newcommand{\mxii}{m_{\xi'}}

\newcommand{\meta}{m_{\eta}}
\newcommand{\eeta}{\boldsymbol{\eta}}
\renewcommand{\th}{\mathrm{th}}

\newcommand{\lgcp}{\mathrm{LGCP}}

\newcommand{\lf}{\mathrm{lf}}

\begin{document}
\title{On a generalization of Matérn hard-core processes with applications to max-stable processes}

\author{M.~Dirrler\footnote{Universit\"at Mannheim, A5,6 68161 Mannheim, Germany, Email address: mdirrler@mail.uni-mannheim.de},\, 
	M.~Schlather\footnote{Universit\"at Mannheim, A5,6 68161 Mannheim, Germany, Email address: schlather@math.uni-mannheim.de}}
\maketitle

\thispagestyle{empty}

\begin{abstract}
The Matérn hard-core processes are classical examples for point process models obtained from (marked) Poisson point processes. Points of the original Poisson process are deleted according to a dependent thinning rule, resulting in a process whose points have a prescribed hard-core distance.
We present a new model which encompasses recent approaches. It generalizes the underlying point process, the thinning rule and the marks attached to the original process. The new model further reveals several connections to mixed moving maxima processes, e.g.\ a process of visible storm centres.
\end{abstract}

{\small
	\phantom{\textit{2010 MSC}:} {Secondary 86A32, 60G70} }

\section{Introduction}

Point process models obtained by dependent thinning of homogeneous Poisson point processes have been extensively examined during the last decades. The Matérn hard-core processes \citep{matern1960hardcore} are classical examples for such processes, where the thinning probability of an individual point depends on the other points of the original point pattern. The Matérn models and slight modifications of them are applied to real data in various branches, for instance ecological science \citep{stoyan87materngeneral,picard2005}, geographical analysis \citep{stoyan88desertedvillages} and computer science \citep{ibrahim}.

There already exist several extensions of Matérns models \citep{kuronen2013}, concerning the hard-core distance \citep{stoyan85onmatern,mansson2002maternconvexgrains}, the thinning rule \citep{teichmann2013hardcore} or the generalization the underlying Poisson process \citep{andersen2015coxhardcore}.
We present a new model which encompasses all these approaches and further generalizes the underlying point process, the thinning rule and the marks attached to the original process.\\
In Section~\ref{sec:prelim}, we shortly review the Matérn hard-core process and state details on Palm calculus which will be used throughout the paper. Our general model is defined in Section~\ref{sec:general_model}. We restrict the underlying ground process to a log Gaussian Cox process in Section~\ref{sec:coxmodel} and calculate first and second order properties for this model. In Section~\ref{sec:appl_m3}, we establish a connection between our model and mixed moving maxima (M3) processes. We prove that a process based on our model is in the max-domain of attraction of an M3 process and further calculate first and second order properties of a point process which can be recovered from an observation of the M3 process itself.

\section{Preliminaries}\label{sec:prelim}
In this paper, we regard a point process as random countable subset of a complete separable metric space $S$. For $\varphi\subset S$ we denote by $n(\varphi)$ the cardinality of $\varphi$. We define the space of locally finite subsets of $S$ by
$$N_{\lf}=\{\varphi\subset S: n(\varphi\cap B)<\infty,\ \text{for all bounded } B\subset S\}$$
and the corresponding $\sigma$-algebra
$$\mathcal{N}_{\mathrm{lf}}=\sigma(\{\varphi\in N_{\mathrm{lf}}:n(\varphi\cap B)=m\}:\ B\subset S\ \text{bounded and } m\in\N).$$
A point process $\Phi$ is a random variable taking values in $(N_{\lf},\mathcal{N}_{\mathrm{lf}})$.

\vspace{2mm}

\textbf{Matérn hard-core processes and Palm calculus.} 
Matérn introduced point process models which are obtained from a stationary Poisson process $\Phi$ on $S=\R^d$ with intensity $\lambda$, by a dependent thinning method \citep{matern1960hardcore}. In the Matérn I model, all points $\xi\in\Phi$ that have neighbours within a deterministic hard-core distance $R$ are deleted. The Matérn II model considers a marked point process $\Phi_M$ where each point $\xi\in\Phi$ is independently endowed with a random mark $m_{\xi}\sim \mathcal{U}[0,1]$. A point $(\xi,\mxi)\in \Phi_M$ is retained in the thinned process $\Phi_{\th}$ if the sphere $B_R(\xi)$ contains no points $\xi'\in\Phi\setminus\{\xi\}$ with $m_{\xi'}<m_{\xi}$. That is,
\begin{align*}
\Phi_{\mathrm{th}}&=\{(\xi,\mxi)\in\Phi_M:m_{\xi}<m_{\xi'},\ \forall \xi'\in\Phi\cap B_R(\xi)\setminus\{\xi\}\}\\
&=\{(\xi,\mxi)\in\Phi_M:f_{\mathrm{th}}(\Phi_M;\xi,m_{\xi})=1\}
\end{align*}
with the \textit{thinning function}
$$f_{\mathrm{th}}(\Phi_M;\xi,m_{\xi})=\prod_{(\xi',m_{\xi'})\in \Phi_M\setminus\{(\xi,m_{\xi})\}}(1-\1_{\xi'\in B_R(\xi)}\1_{m_{\xi'}<m_{\xi}}).$$
It is of particular interest to compute the probability that a given point $(\xi,\mxi)\in\Phi_M$ is retained in the thinned process. This probability can be calculated using Palm calculus \citep{mecke1967,moeller2004ppbook, daley2008pp,chiu_stoyan}, which we briefly summarize in the following. 

The \textit{reduced Campbell measure} $C^!$ for a point process $\Phi$ on $S$ is a measure on $S\times N_{\mathrm{lf}}$ defined by
$$C^!(D)=\E\sum_{\xi\in\Phi}\1_{(\xi,\Phi\setminus{\xi})\in D},\quad D\subset S\times N_{\mathrm{lf}}.$$
Let the intensity measure $\mu$ be $\sigma$-finite. Then the Campbell measure is, in its first component, absolutely continuous with respect to $\mu$. Its Radon-Nikodym density $P^!_{\xi}$ is called \textit{reduced Palm distribution}. Therefore

$$C^!(B\times F)=\int_BP^!_{\xi}(F)\d\mu(\xi),\quad B\subset S, F\in \mathcal{N}_{\mathrm{lf}}$$

and for non-negative functions $h:S\times N_{\mathrm{lf}}\rightarrow [0,\infty)$

$$\E\sum_{\xi\in\Phi}h(\xi,\Phi\setminus\{\xi\})=\int\int h(\xi,\eta)\d P^!_{\xi}(\eta)\d\mu(\eta).$$

Hence $P^!_{\xi}$ can be interpreted as the conditional distribution of $\Phi\setminus{\{\xi\}}$ given $\xi\in\Phi$. Thereby the retaining probability of a point $(\xi,\mxi)\in\Phi_M$ equals

$$r(\xi,m_{\xi})=\int_{{M}_{\mathrm{lf}}}f_{\mathrm{th}}(\varphi;\xi,m_{\xi})~P_{\xi,\mxi}^!(\d\varphi)$$
where $M_{\lf}$ is the suitably defined space of point configurations of the marked process $\Phi_M$. We give the full definition of $M_{\lf}$ and its $\sigma$-algebra in the next section.
The \textit{generating functional} of a point process $\Phi$ is defined as
$$G_{\Phi}(u)=\E\prod_{\xi\in\Phi}u(\xi)$$
for functions $u: S\rightarrow [0,1]$, see \citep{westcott}.
The Palm distribution $P^!_{\xi,\mxi}$ equals the distribution of $\Phi_M$ since $\Phi_M$ is a Poisson process - see Example 4.3 in \cite{chiu_stoyan}. As a consequence of this, $r(\xi,m_{\xi})$ is the generating functional of $\Phi_M$ evaluated at $f_{\th}$. Therefore
\begin{align*}
r(\xi,\mxi)=\exp(-\lambda |B_R(o)|\cdot \mxi)
\end{align*}
and the intensity of the thinned process equals 
\begin{align*}
\lambda_{\mathrm{th}}=\lambda\int_0^1 r(\mxi)~\d \mxi= |B_R(o)|^{-1}(1-\exp(-\lambda |B_R(o)|)).
\end{align*}

The Palm distribution of a general point process is more difficult to handle, however it can be explicitly calculated for many Cox process models \citep{moeller03shotnoisecox,coeurjolly2015}. Besides, \cite{mecke1967} indicates how to calculate the Palm distribution of an infinitely divisible point process. 
A point process $\Phi$ is called infinitely divisible if, for all $n\in\N$ there exist iid.\ processes $\Phi_1,\dots, \Phi_n$ such that $\Phi\stackrel{d}{=}\Phi_1+\dots+\Phi_n$.

\vspace{2mm}

\section{Generalizing the Matérn hard-core processes}\label{sec:general_model}
We present a new point process model, obtained by dependent thinning of a ground process $\Phi$, which generalizes the Matérn model in several ways. We therefore call the new model \textit{generalized Matérn model}.

Suppose that $\Phi$ is a locally finite point process on $S$. Each point $\xi$ of $\Phi$ is independently attached with a random mark $m_{\xi}$. We allow these marks to be continuous functions from $S$ to $\R$, i.e.\ an element of the space of continuous functions $\M= {C}(S,\R)$ with law $\nu$.
Then,
$$\Phi_M=\{(\xi,m_{\xi}):\xi\in \Phi\}$$
is a marked point process, i.e.\ a mapping into $(M_{\lf},\mathcal{M}_{\lf})$, with the set of point configurations

$$M_{\lf}=\left\{\varphi\subset S\times \mathbb{M}:\{\xi\in S,(\xi,\mxi)\in\varphi\}\in N_{\mathrm{lf}} \text{ and } (\xi,\mxi),(\xi,\mxi')\in\varphi \Rightarrow \mxi=\mxi'\right\}
$$
and its $\sigma$-algebra $\mathcal{M}_{\mathrm{lf}}$ which is defined analogous to $\mathcal{N}_{\lf}$.
The Bernoulli random variable $\tau_{\Phi_M;\xi,m_\xi}$ shall indicate whether a point of $\Phi_M$ is retained in the thinned process. We define the thinned marked process 
\begin{align}\label{eq:thinned}
\Phi_{\mathrm{th}}=\{(\xi,m_{\xi})\in\Phi_M:(\xi,m_{\xi})\in \Phi_M, \tau_{\Phi_M;\xi,m_\xi}=1\},
\end{align}
which we call \textit{generalized Matérn process}, and the thinned ground process
\begin{align}\label{eq:thinnedground}
\Phi_{\th}^0=\{\xi:(\xi,\mxi)\in\Phi_{\th}\}.
\end{align}
The success probability of $\tau_{\Phi_M;\xi,m_\xi}$ equals the \textit{thinning function} 
$$f_{\mathrm{th}}(\Phi_M;\xi,m_{\xi})=\prod_{(\xi',\mxi')\in \Phi_M}(1-\zeta(\xi,m_{\xi},\xi',m_{\xi'})p(\xi,m_{\xi},\xi',m_{\xi'})).$$
Here, $\zeta:(S,\M)^2\rightarrow\{0,1\}$ is a measurable function which we call \textit{competition function} and which specifies the inferior points which are endangered to be deleted. We call a point $\xi$ \textit{inferior} if $\zeta(\xi,\mxi,\xi',\mxi')=1$ for some $(\xi',\mxi')\in\Phi_M\setminus\{\xi,\mxi\}$. Likewise, $p:(S,\M)^2\rightarrow[0,1]$ is a measurable function determining the probability that an inferior point is deleted.
We henceforth fix $$\zeta(\xi,\mxi,\xi,\mxi)=1, \quad p(\xi,\mxi,\xi,\mxi)=1-p_0\in[0,1]$$ and thereby include independent $p_0$-thinning in our model. In order to simplify notation, we will use abbreviations like $\xxi=(\xi,\mxi)$ and $\zeta(\xxi,\xxi')=\zeta(\xi,\mxi,\xi',\mxii)$ throughout the paper.

%

\begin{table}[t]
	\centering
	\begin{tabular}{l|c|c|c}
		&$p$&$\mxi$&$\zeta$\\\hline
		Matérn I& 1 &- &$\1_{\xi'\in B_R(\xi)}$\\
		Generalized Matérn I&$(1-\|\xi-\xi'\|/R)_+$&-&$\1_{\xi'\in B_R(\xi)}$\\
		Matérn II& 1&$\mathcal{U}[0,1]$&$\1_{\xi'\in B_R(\xi)}\1_{\mxi'<\mxi}$\\
		
		Generalized Matérn II&$(1-\|\xi-\xi'\|/R)_+$&$\mathcal{U}[0,1]$&$\1_{\xi'\in B_R(\xi)}\1_{\mxi'<\mxi}$\\
	\end{tabular}
	\caption{Let $S=\R^d$. The classical Matérn models are obtained as special case of our general model. We call the models resulting from Matérn I or II by including an additional stochastic thinning, \textit{generalized Matérn I} and \textit{II} model respectively. }
	\label{tab:models}
\end{table}

\begin{Example}
	The Matérn hard-core models I and II can be easily derived from our model, see Table~\ref{tab:models}. There, we also give generalizations where $p\neq 1$.
\end{Example}

\begin{figure}[H]
	\centering
	\begin{minipage}[b]{6.5 cm}
		\includegraphics[width=1\textwidth]{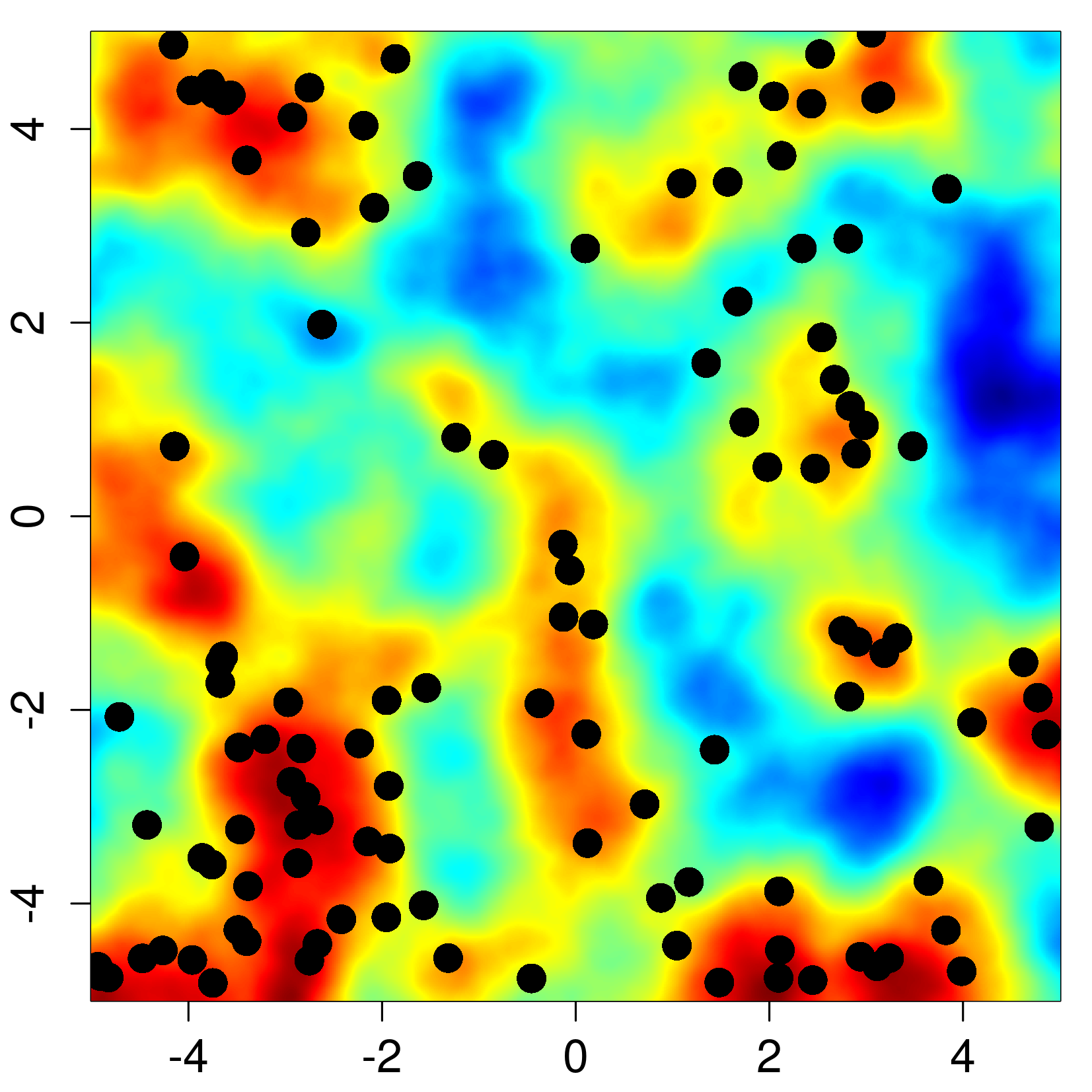}
		\includegraphics[width=1\textwidth]{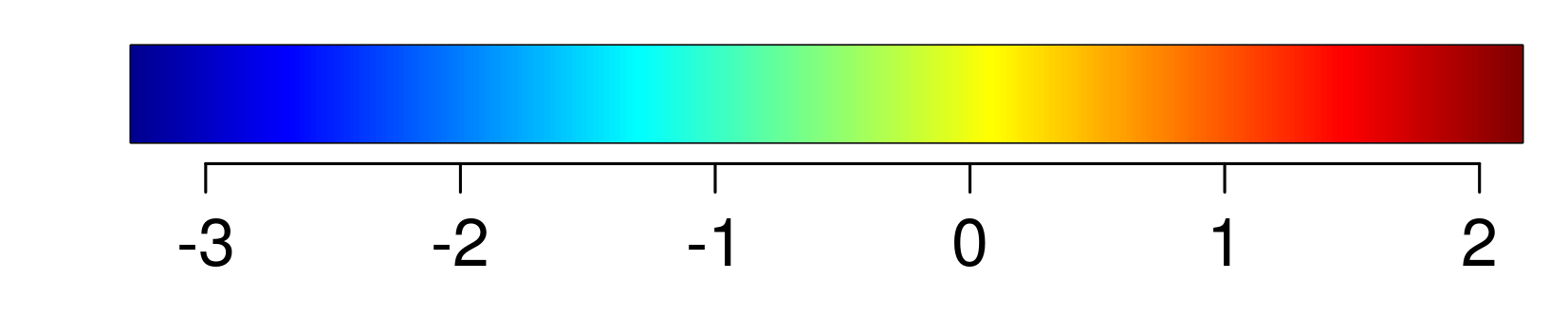}
	\end{minipage}
	\begin{minipage}[b]{6.5cm}
		\includegraphics[width=1\textwidth]{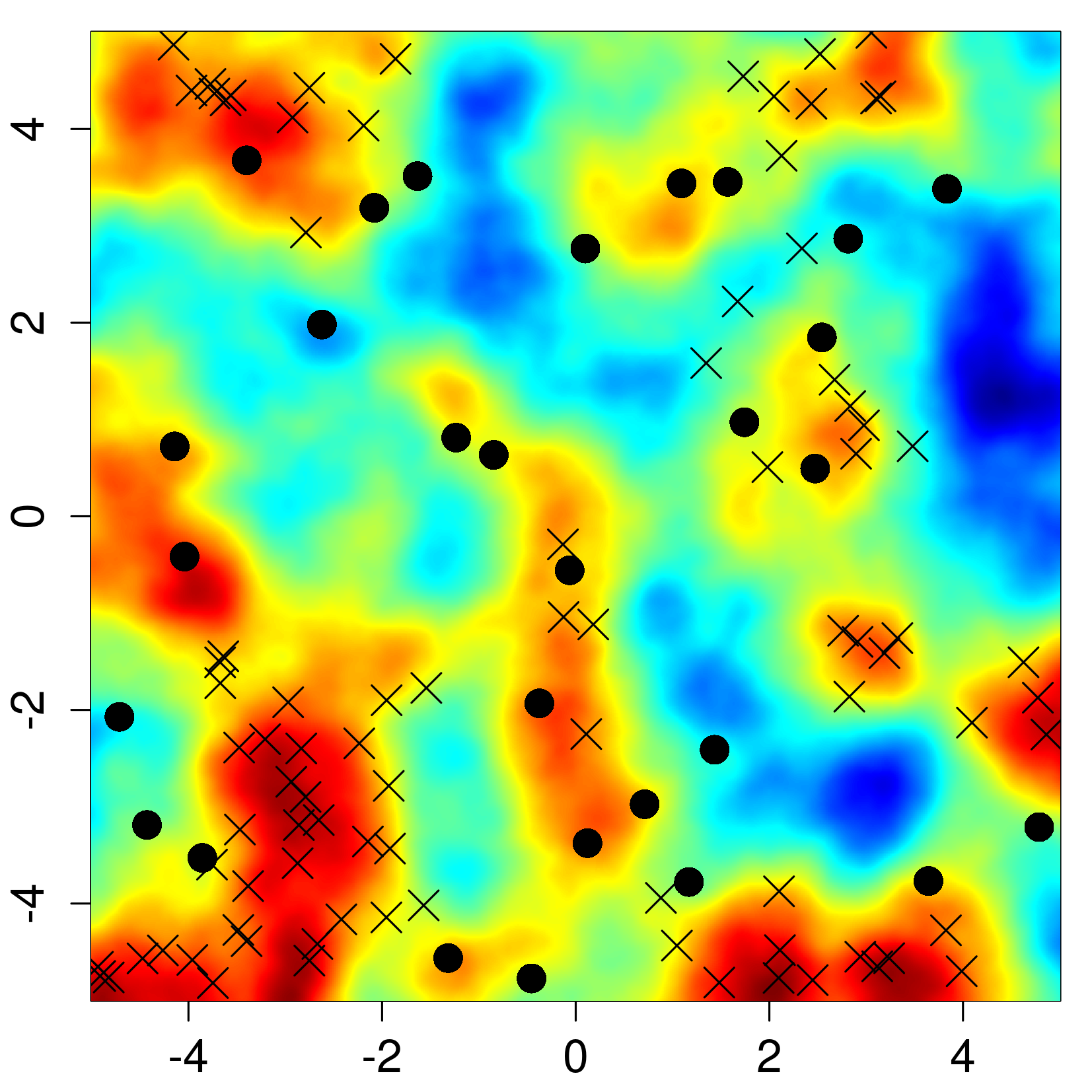}
		\includegraphics[width=1\textwidth]{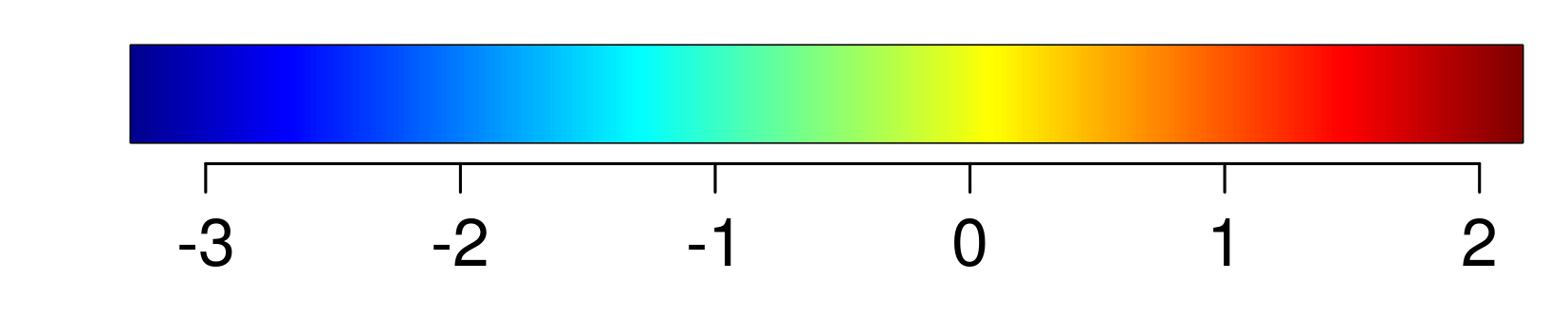}
	\end{minipage}
	\begin{minipage}[b]{6.5cm}
		\includegraphics[width=1\textwidth]{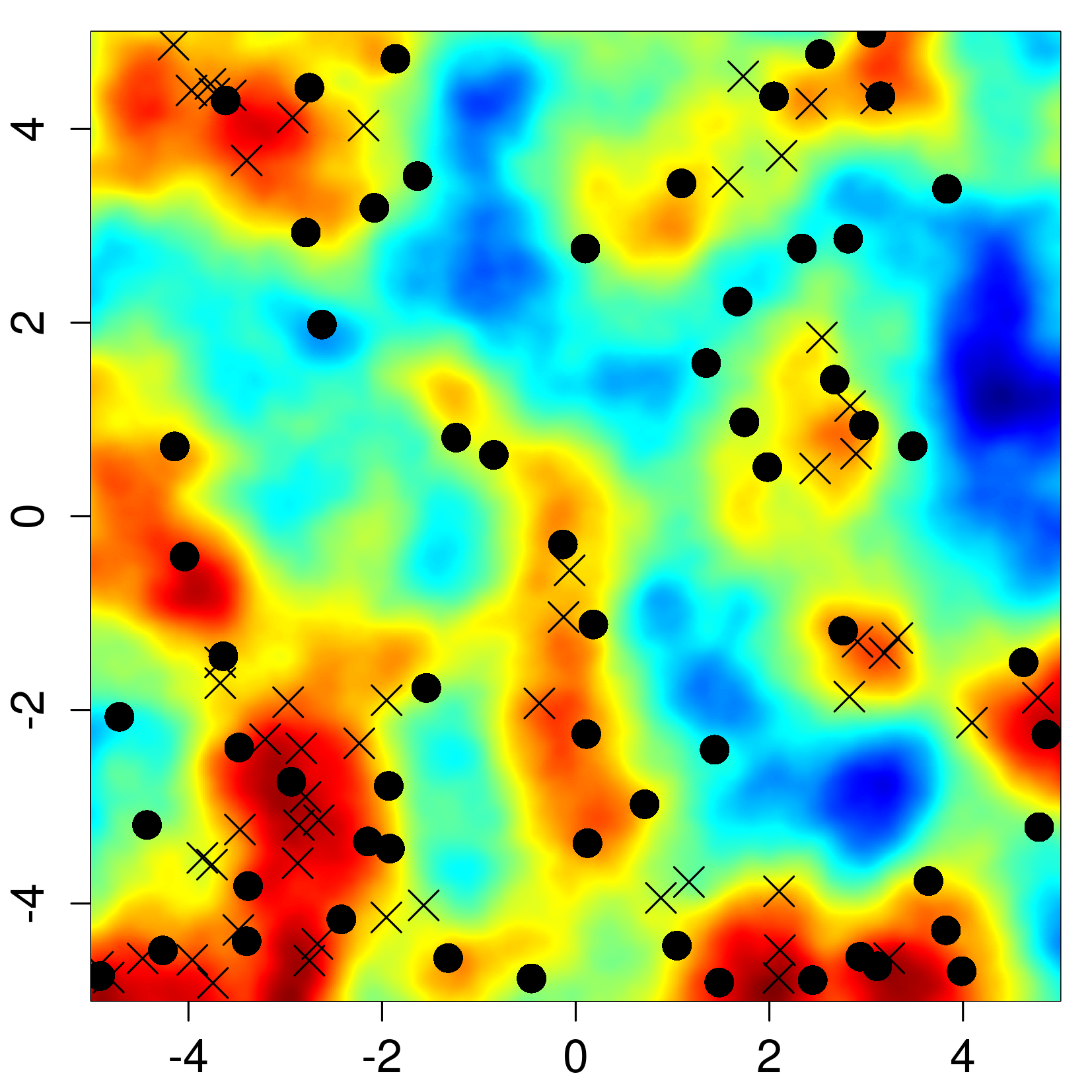}
		\includegraphics[width=1\textwidth]{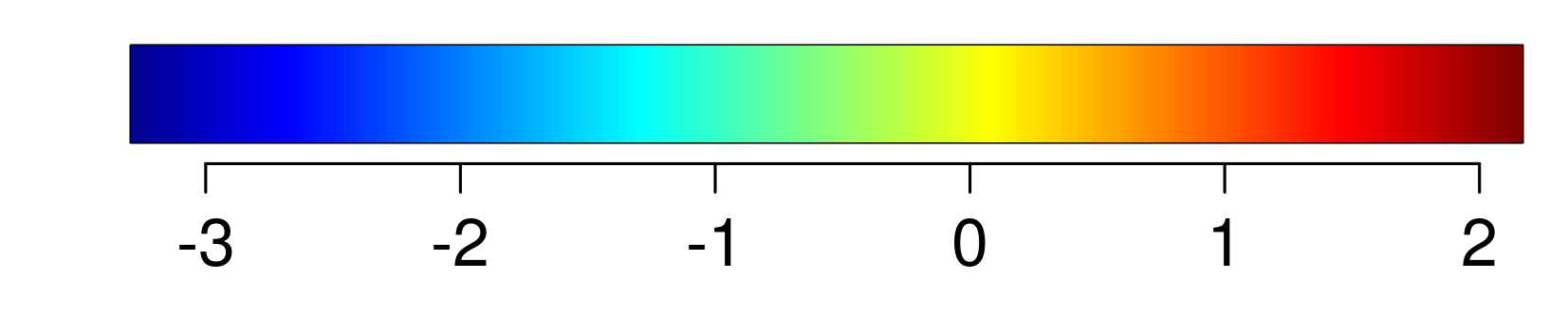}
	\end{minipage}
	\begin{minipage}[b]{6.5cm}
		\includegraphics[width=1\textwidth]{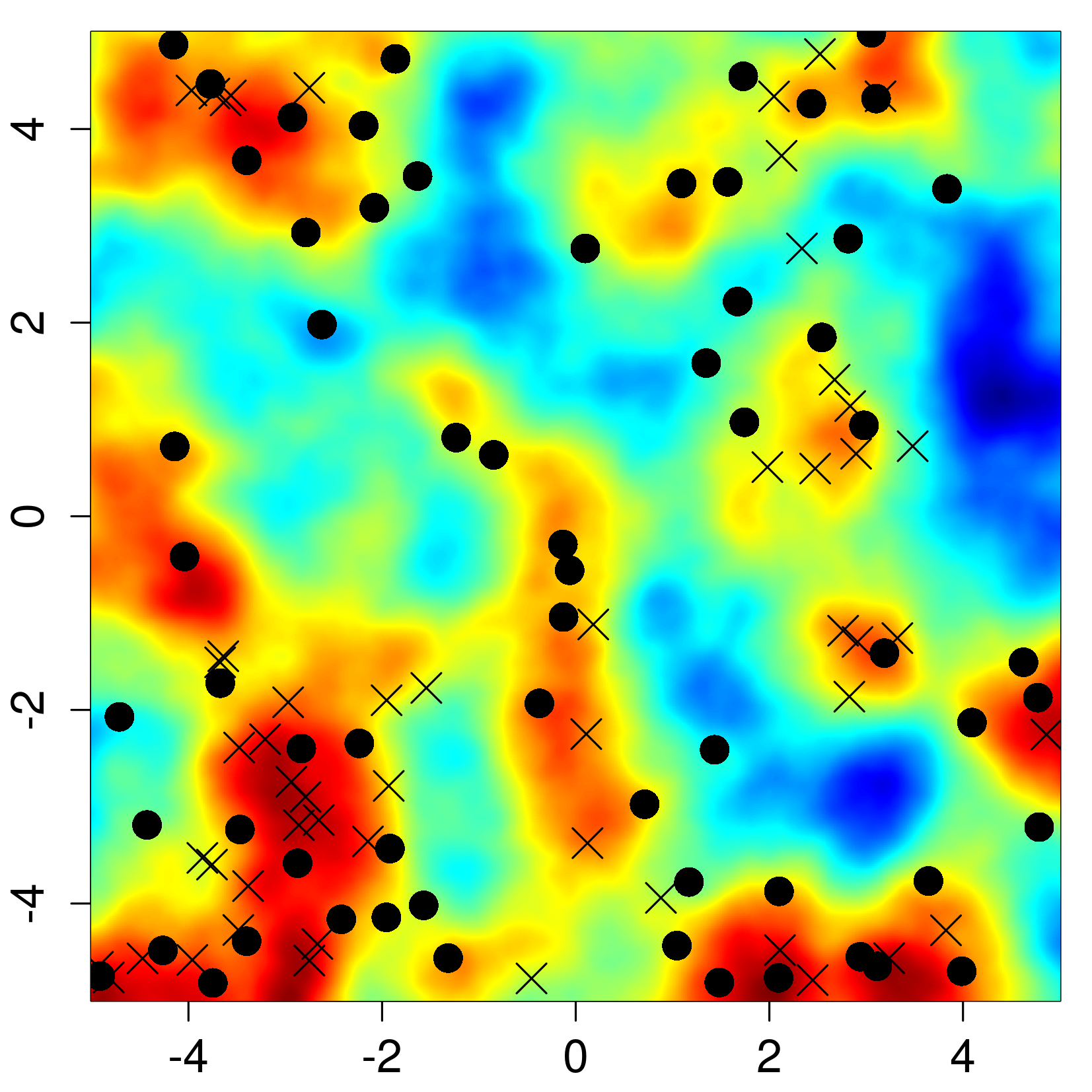}
		\includegraphics[width=1\textwidth]{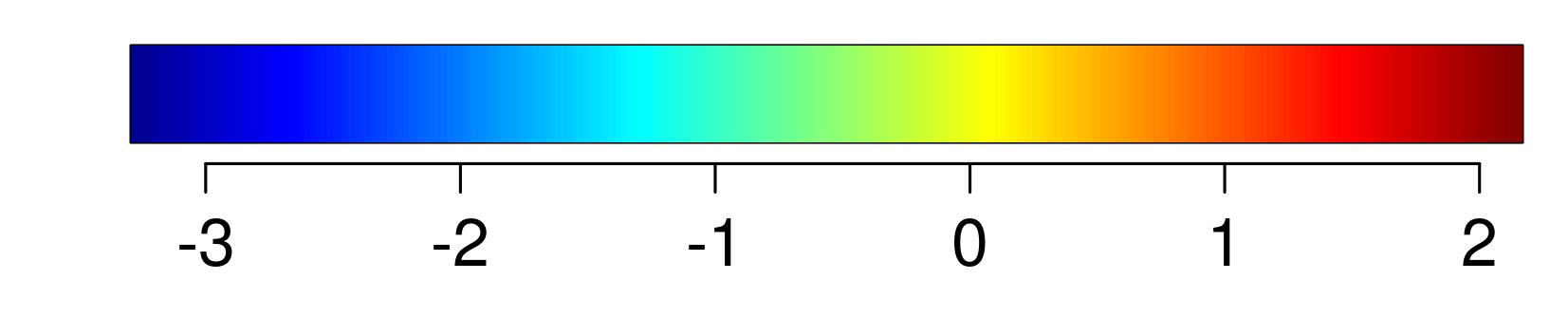}
	\end{minipage}
	\caption{Plot of the original Poisson process $\Phi$ and the underlying intensity function (upper left). Thinned points $\Phi_{\mathrm{\mathrm{th}}}^0$ of a generalized Matérn I model (upper right) and a generalized Matérn II model (lower left) with $R=1$ and $p_{\xi'}(\xi)=\max(0,1-\|\xi-\xi'\|)$.	The last plot shows the thinned points of a generalized hard-core process with competition function $\zeta(\xxi,\xxi')=\1_{m_{\xi'}(\xi)>m_{\xi}(\xi)}$, random mark functions  $m_{\xi}(\cdot)=u\cdot \varphi(\cdot)$, $u\sim\mathcal{U}[0,1]$ with the two-dimensional standard-normal density $\varphi$ and thinning probability $p(\xxi,\xxi')=\max(0,1-\|\xi-\xi'\|)$.}
	\label{fig:plots}
\end{figure}

\begin{Example}
	A further generalization of the Matérn I model in $S=\R^d$ was presented in \cite{teichmann2013hardcore}. According to their thinning rule, a point $\xi$ of the ground process $\Phi$ is retained with probability $$p_0\prod_{\xi'\in \Phi_M\setminus\{\xi\}}(1-f(\|\xi-\xi'\|)),$$
	with $p_0\in(0,1]$ and some deterministic function $f:[0,\infty)\rightarrow [0,1]$. This equals our model with the choice $\zeta\equiv 1$ and $p(\xi,\mxi,\xi',\mxi')=f(\|\xi-\xi'\|)$.
\end{Example}

\begin{Example}
	Consider now $S=\R^d$, marks $m_{\xi}$ in $\M=\R^{\{0,1\}}$ with $m_{\xi}(0)\sim \mu$ and $m_{\xi}(1)\sim \nu$ for probability measures $\mu$ and $\nu$. Let $\zeta(\xi,m_{\xi},\xi',m_{\xi'})=\1_{m_{\xi}(0)\geq m_{\xi'}(0)}$ and $p(\xi,m_{\xi},\xi',m_{\xi'})=f(\|\xi-\xi'\|,m_{\xi}(1),m_{\xi'}(1))$. Then 
	$$f_{\mathrm{th}}(\Phi_M;\xi,m_{\xi})=p_0\prod_{\boldsymbol{\xi'}\in\Phi_M\setminus\{\boldsymbol{\xi}\}}\big[1-\1_{m_{\xi}(0)\geq m_{\xi'}(0)}f\big(\|\xi-\xi'\|,m_{\xi}(1),m_{\xi'}(1)\big)\big].$$
	This model was presented by \cite{teichmann2013hardcore} as an extension of the Matérn II model.
\end{Example}

\begin{Example}\label{ex:huetchen}
	Let $\Phi$ be an inhomogeneous Poisson process in $\R^d$, attached with random mark functions  $m_{\xi}(\cdot)=u\cdot \varphi(\cdot)$, $u\sim\mathcal{U}[0,1]$ with the $d$-dimensional standard-normal density $\varphi$. Consider the competition function $\zeta(\xxi,\xxi')=\1_{m_{\xi'}(\xi)>m_{\xi}(\xi)}$ and $p(\xxi,\xxi')=\max(0,1-\|\xi-\xi'\|)$. This leads to a soft-core model where inferior points are the more likely to be thinned the closer they are to superior points. See Figure~\ref{fig:plots} for a plot of this model in $d=2$ and Figure~\ref{fig:huetchen} for a plot of arbitrary points $(\xi,\mxi)$ and $(\eta,\meta)$ with $d=1$.
\end{Example}

\begin{figure}[H]
	\centering
	\includegraphics[width=0.35\textwidth]{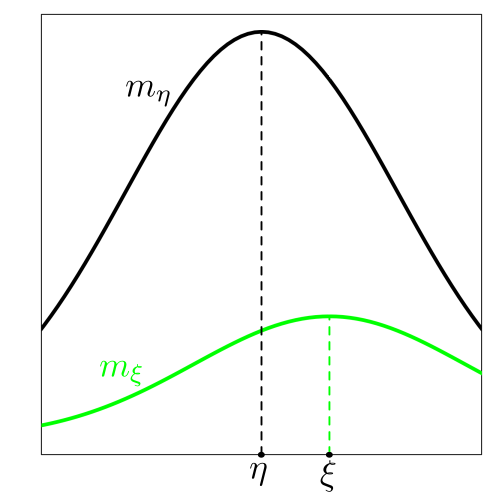}
	\caption{Arbitrary points $(\xi,\mxi)$ and $(\eta,\meta)$ of $\Phi_M$ from Example \ref{ex:huetchen} with $d=1$. Since $\meta(\xi)>\mxi(\xi)$, the point $(\xi,\mxi)$ is inferior to $(\eta,\meta)$ and hence endangered to be thinned with probability $p(\xi,\mxi,\eta,\meta)=\max(0,1-|\xi-\eta|)$.}
	\label{fig:huetchen}
\end{figure}

\section{Generalized Matérn model based on log Gaussian Cox processes}\label{sec:coxmodel}
Let $P_{\xxi}^!(\cdot)$ be the reduced Palm distribution of $\Phi_M$, that is $P_{\xxi}^!(\cdot)$ is a probability measure on $(M_{\mathrm{lf}},\mathcal{M}_{\mathrm{lf}})$ for each $\xxi\in S\times \M$.
The retaining probability of a point $\xi\in \Phi$ with mark function $m_{\xi}$ can then be calculated by
$$r(\xi,m_{\xi})=\int_{{M}_{\mathrm{lf}}}f_{\mathrm{th}}(\varphi\cup (\xi,\mxi);\xi,m_{\xi})~P_{\xxi}^!(\d\varphi).$$
Since our new model is defined in a rather general setting, reasonable restrictions are needed in order to calculate the reduced Palm distribution of $\Phi_M$ and thereby first and seconder order properties of $\Phi_{\th}$. %
We henceforth assume that $\Phi$ is a log Gaussian Cox process, though all results may be derived in a similar way for other Cox processes or infinitely divisible point processes \citep{mecke1967}, when $P_{\xxi}^!$ is known.

Let $\Psi=\exp(W)$ be the random intensity function of $\Phi$ where $W$ is a Gaussian random field with mean function $\mu$ and covariance function $C$. We write $\Phi\sim\lgcp(\mu,C)$ for short.

\begin{Proposition}\label{prop:FOI}
	Let $\Phi\sim\lgcp(\mu,C)$ and $h(\xxi,\xxi')=1-\zeta(\xxi,\xxi')p(\xxi,\xxi')$, then the retaining probability is
	$$r(\xi,m_{\xi})=\E\prod_{\xxi'\in\widetilde{\Phi}_M} h(\xxi,\xxi'),$$
	with $\widetilde{\Phi}\sim\lgcp(\widetilde{\mu},C)$ and $\widetilde{\mu}(\cdot)=\mu(\cdot)+C(\cdot,\xi)$.
	Furthermore, if $\widetilde{\Psi}$ is the random intensity function of $\widetilde{\Phi}$, the first order intensity of $\Phi_{\th}^0$ is given by
	\begin{align}
	\label{eq:FOI}
	\rho_{th}(\xi)=p_0\rho_{\Phi}(\xi)\int_{\M}\E_{\widetilde{\Psi}}\exp\left(-\int_{S}\int_{\M}\zeta(\xxi,\xxi')p(\xxi,\xxi')\widetilde{\Psi}(\xi')\nu(dm_{\xi'})\d\xi'\right)\nu(\d m_{\xi}),
	\end{align}
	where $\rho_{\Phi}$ is the intensity of $\Phi$.
\end{Proposition}

\begin{proof}
	The reduced Palm distribution $P_{\xi}^!$ of $\Phi$ equals the distribution of $\widetilde{\Phi}\sim\lgcp(\widetilde{\mu},C)$ since $\Phi\sim\lgcp(\mu,C)$ - see Proposition 1 in \cite{coeurjolly2015}. Therefore
	\begin{align*}r(\xi,m_{\xi})=\ &\int_{{M}_{\mathrm{lf}}}f_{\mathrm{th}}(\varphi\cup \{\xxi\};\xxi)~P_{\xxi}^!(\d\varphi)\\
	=&\  \E\prod_{\xxi'\in\widetilde{\Phi}_M\cup\{\xxi\}} [1-\zeta(\xxi,\xxi')p(\xxi,\xxi')]\\
	=&\ (1-\zeta(\xxi,\xxi)p(\xxi,\xxi))\E\prod_{\xxi'\in\widetilde{\Phi}_M} [1-\zeta(\xxi,\xxi')p(\xxi,\xxi')]\\
	=&\  p_0\ \E\prod_{\xxi'\in\widetilde{\Phi}_M} h(\xxi,\xxi')\\
	=& \  p_0\ \E_{\widetilde{\Psi}} \exp\left(-\int_{S\times\M} \big(1-h(\xxi,\xxi')\big)\widetilde{\Psi}(\xi')\d\xi'\nu(\d m_{\xi'})\right),
	\end{align*}
	where the last equality follows from calculating the generating functional of $\widetilde{\Phi}_M$.
	
\end{proof}

\begin{Proposition}\label{prop:soi_general}
	Consider $\widetilde{\Phi}\sim\lgcp(\tilde{\mu},C)$ with $\widetilde{\mu}(\cdot)=\mu(\cdot)+C(\cdot,\xi)+C(\cdot,\eta)$ and let $\rho_{\Phi}^{(2)}$ be the second order intensity of $\Phi$. Then, the second order intensity of the thinned process $\Phi_{\th}^0$ equals
	\begin{align*}\rho_{\mathrm{th}}^{(2)}(\xi,\eta)=\rho^{(2)}_{\Phi}(\xi,\eta)p_0^2&\int_{\M}\int_{\M}h(\xxi,\eeta)h(\eeta,\xxi)\\
	&\E_{\widetilde{\Psi}} \exp\left(-\int_{S\times\M} \big(1-h(\xxi,\xxi')h(\eeta,\xxi')\big)\widetilde{\Psi}(\xi')\d\xi'\nu(\d\mxi')\right)\nu(\d\mxi)\nu(\d\meta).
	\end{align*}
\end{Proposition}

\begin{proof}
	The probability that none of the two arbitrary points $(\xi,\mxi)$ and $(\eta,\meta)$ is deleted by any point of the point configuration $\varphi\in M_{\mathrm{lf}}$ is
	$$f^{(2)}_{\th}(\varphi;\xxi,\eeta)=\prod_{\xxi'\in \varphi}(1-\zeta(\xxi,\xxi')p(\xxi,\xxi'))(1-\zeta(\eeta,\xxi')p(\eeta,\xxi')).$$
	
	Thus, the probability that $(\xi,\mxi),(\eta,\meta)\in\Phi_M$ are retained in $\Phi_{\th}$ equals
	\begin{align*}
	r(\xxi,\eeta)=\int_{M_{\mathrm{lf}}}f_{\th}^{(2)}(\varphi\cup\{\xxi,\eeta\};\xxi,\eeta)P^!_{\xxi,\eeta}(\d\varphi),
	\end{align*}
	where $P^!_{\xxi,\eeta}$ is the two-point reduced Palm distribution of $\Phi_M$, which is also
	the distribution of a log Gaussian Cox process $\widetilde{\Phi}\sim\lgcp(\tilde{\mu},C)$ with $\widetilde{\mu}(\cdot)=\mu(\cdot)+C(\cdot,\xi)+C(\cdot,\eta)$ - see again Proposition 1 in \cite{coeurjolly2015}.
	Therefore 
	\begin{align*}
	&\int_{M_{\mathrm{lf}}}f_{\th}^{(2)}(\varphi\cup\{\xxi,\eeta\};\xxi,\eeta)P^!_{\xxi,\eeta}(\d\varphi)\\
	=&\  \E\prod_{\xxi'\in\widetilde{\Phi}_M\cup\{\xxi,\eeta\}} \bigg(1-\zeta(\xxi,\xxi')p(\xxi,\xxi')\bigg)\bigg(1-\zeta(\eeta,\xxi')p(\eeta,\xxi')\bigg)\\
	=&\   p_0^2h(\xxi,\eeta)h(\eeta,\xxi)\ \E\prod_{\xxi'\in\widetilde{\Phi}_M} h(\xxi,\xxi')h(\eeta,\xxi')\\
	=& \  p_0^2h(\xxi,\eeta)h(\eeta,\xxi)\ \E_{\widetilde{\Psi}} \exp\left(-\int_{S\times\M} \bigg(1-h(\xxi,\xxi')h(\eeta,\xxi')\bigg)\widetilde{\Psi}(\xi')\d\xi'\nu(\d\mxi)\right).
	\end{align*}
	
\end{proof}

\section{Application to mixed moving maxima processes}\label{sec:appl_m3}

We establish a connection between the generalized Matérn model and mixed moving maxima processes in this section. In the first subsection, we choose a specific thinning function and prove that a process based on the corresponding generalized Matérn process \eqref{eq:thinned} converges to known (conditional) mixed moving maxima processes. We slightly modify this thinning function in the second subsection to obtain a process whose first and second order properties can be derived and whose points can be recovered from observations of the mixed moving maxima process itself under rather mild assumptions.\\

\textbf{General framework.} Let $S=\R^d$, $K\subset S$ be compact and let $X$ be a stochastic process whose paths are almost surely in $\X=C(S,\R)$ and which fulfils the condition 
\begin{align}\label{eq:Xcond2}
\E_X\int_S\sup_{t\in K}X(t-\xi)~\d\xi<\infty.
\end{align}
Then a mixed moving maxima process \citep{smith1990maxstable} is defined by
\begin{align}\label{def:M3} 
Z(t)=\bigvee_{(s,u,X)\in\Theta} u X(t-s), \qquad t \in S,
\end{align}
where $\Theta$ is a Poisson process on $S\times (0,\infty]\times \mathbb{X}$ with directing measure
\begin{align*}
\d\lambda(s,u,X)= \d s \, u^{-2} \d u \, \d \P_X.
\end{align*}
Further, we assume that $\Psi$ is a non-negative process with 
\begin{align}\label{eq:Xcond1}
\E_{\Psi}\E_X\int_S\sup_{t\in K} X(t-\xi)\Psi(\xi)\d\xi<\infty.
\end{align}
Then the Cox extremal process \citep{d16} is defined in an analogous way
\begin{align}\label{def:Coxextremal}
Y(t)=\bigvee_{(s,u,X)\in\widetilde{\Theta}}uX(t-s),
\end{align}
where $\widetilde{\Theta}$ is a Cox process directed by the random measure $\d\Lambda(s,u,X)=\Psi(s)\d su^{-2}\d u\d\P_X$.

We henceforth consider $\Phi\sim\lgcp(\mu-\log(\tau),C)$ with random intensity function $\tau^{-1}\Psi$ and $\tau>0$. Each point $\xi$ of $\Phi$ is independently attached with a random mark function $\mxi(\cdot)=U_{\xi}X_{\xi}(\cdot-\xi)$, where $U\sim \tau u^{-2}\1_{(\tau,\infty)}\d u$ and $X\sim\d\P_X$. We assume that for each path $X(\omega,\cdot)$ of $X(\cdot)$ there exist monotonously decreasing functions $f_{\omega}$ and $g_{\omega}$ such that
\begin{align}\label{eq:Xmonotone}
g_{\omega}(\|t\|)\leq X(\omega,t)\leq f_{\omega}(\|t\|),\quad \forall t\in S,
\end{align}
and $g_{\omega}(0)=X(\omega,0)=f_{\omega}(0)$. 

\subsection{Matérn extremal process}

We choose the competition function $\zeta$ such that a point $\xi\in\Phi$ is deleted if its corresponding mark function $\mxi$ is - at each point - strictly smaller than the mark function $\mxii$ of some other point $\xi'\in\Phi$. That is, the thinning function can be written as	
\begin{align}\label{eq:thinExHardCore}
f_{\th}(\Phi_M;\xi,\mxi)=\prod_{\xxi'\in\Phi_M}\left[1-\1_{u_{\xi'}>\sup_{t\in S}u_{\xi}X_{\xi}(t-\xi)X_{\xi'}(t-\xi')^{-1}}\right]
\end{align}
and the process resulting from dependent thinning is $$\Phi_{\th}=\{(\xi,\mxi)\in\Phi_M:f_{\th}(\Phi_M;\xi,\mxi)=1\}.$$

We introduce the \textit{Matérn extremal process} defined by
\begin{align}
\label{def:extremal_hardcore}
\Pi(t)=\bigvee_{(\xi,\mxi)\in\Phi_{\th}}\mxi(t).
\end{align}
The $\mxi$ are closely related to the extremal functions of $\Phi_M$ introduced by \cite{dombry2013contpoints}, though the set $\Phi_{\th}$ is usually much larger than the set of extremal functions. The intensity of $\Phi_{\th}$ is finite, this is a fundamental difference compared to the Cox extremal process $Y$ whose underlying point process $\widetilde{\Theta}$ is infinite. Still, the following lemma shows that the two processes coincide in the limiting case $\tau\rightarrow 0$.

\begin{Lemma}\label{lemma:conv}
	Let the conditions \eqref{eq:Xcond2}, \eqref{eq:Xcond1} and \eqref{eq:Xmonotone} hold true and assume that $\Phi\sim\lgcp(\mu-\log\tau,C)$.
	If $\tau\rightarrow 0$, the convergence
	$$\Pi\rightarrow Y$$
	holds weakly in $C(S,\R)$.
\end{Lemma}
\begin{proof}
	The sample paths of the extremal hard-core process $\Pi$ are continuous, since the marks $\mxi$ are continuous and $\Phi$ is locally finite.
	The finite-dimensional distributions of $\Pi$ are given by
	\begin{align*}
	&\ \P(\Pi(t_1)\leq y_1, \dots, \Pi(t_n)\leq y_n)=\P(U_{\xi}X_{\xi}(t_1-\xi)\leq y_1, \dots, U_{\xi}X_{\xi}(t_n-\xi)\leq y_n, \forall(\xi,\mxi)\in\Phi_{\th})\\
	=& \ \P\left(U_{\xi}\leq\min_{1\leq i\leq n}({y_i}{X_{\xi}(t_i-\xi)^{-1}}),\forall(\xi,\mxi)\in\Phi_{\th}\right)=\P\left(U_{\xi}\leq\min_{1\leq i\leq n}({y_i}{X_{\xi}(t_i-\xi)^{-1}}),\forall(\xi,\mxi)\in\Phi_M\right)\\
	=&\ \E_{\Psi}\exp\left[-\int_{\X}\int_S\int_{\min_{1\leq i\leq n}({y_i}{X_{\xi}(t_i-\xi)^{-1}})} u^{-2}\1_{(\tau,\infty)}(u)~\d u\Psi(\xi)\d\xi\d\P_X \right]\\
	=&\ \E_{\Psi}\exp\left[-\int_{\X}\int_S \max\left(\tau,\min_{1\leq i\leq n}({y_i}{X_{\xi}(t_i-\xi)^{-1}})\right)^{-1}\Psi(\xi)~\d\xi~\d\P_X\right]\\
	=&\ \E_{\Psi}\exp\left[-\int_{\X}\int_S \min\left(1/\tau,\max_{1\leq i\leq n}({y_i}^{-1}{X_{\xi}(t_i-\xi)})\right)\Psi(\xi)~\d\xi~\d\P_X\right]
	\end{align*}
	Hence, with condition \eqref{eq:Xcond1}
	\begin{align*}
	\lim_{\tau\rightarrow 0}\ \P(\Pi(t_1)\leq y_1, \dots, \Pi(t_n)\leq y_n)=\E_{\Psi}\exp\left[-\int_{\X}\int_S \max_{1\leq i\leq n}\left({y_i}^{-1}X_{\xi}(t_i-\xi)\right)\Psi(\xi)~\d\xi~\d\P_X\right]
	\end{align*}
	which equals the finite-dimensional distribution of $Y$, see Remark 3 in \cite{d16}.
	It remains to prove the tightness of $\Pi$, that is
	$$\lim_{\delta\rightarrow 0}\limsup_{\tau\rightarrow 0} \P(\omega_K(\Pi,\delta)>\varepsilon)=0$$
	with an arbitrary compact set $K\subset S$ and
	$$\omega_K(\Pi,\delta)=\sup_{t_1,t_2\in K:\|t_1-t_2\|\leq\delta}\left|\Pi(t_1)-\Pi(t_2)\right|.$$
	This can be shown in the same way as in Theorem 7 of \cite{d16}.

\end{proof}
\begin{Theorem}
	Let the assumptions of Lemma~\ref{lemma:conv} hold true and let $\Psi$ be stationary with $\E\Psi(o)=1$.
	The Matérn extremal process $\Pi$ is in the max-domain of attraction of the mixed moving maxima process $Z$ given by Equation \eqref{def:M3}.	That is, if $\Pi_i$ are iid.\ copies of $\Pi$, the convergence 
	$$n^{-1}\bigvee_{i=1}^n\Pi_i\rightarrow Z,$$
	holds weakly in $C(S,\R)$.
\end{Theorem}

\begin{proof}
	Consider the sequence $\Pi^{(n)}=n^{-1}\bigvee_{i=1}^n \Pi_i$. We have to prove that $\Pi^{(n)}$ is tight, and that its marginal distributions converge to that of $Z$ which are given by
	$$\P(Z(t_1)\leq z_1, \dots, Z(t_m)\leq z_m)=\exp\left[-\int_{\X}\int_S \max_{1\leq i\leq n}\left({y_i}^{-1}X_{\xi}(t_i-\xi)\right)~\d\xi~\d\P_X\right].$$
	The tightness can be derived by similar arguments as in the proof of Theorem~7 in \cite{d16}. The finite-dimensional distributions of $\Pi^{(n)}$ are given by
	
	\begin{align*}
	&\ \P(\Pi^{(n)}(t_1)\leq z_1, \dots, \Pi^{(n)}(t_m)\leq z_m)\\
	=&\ \prod_{i=1}^n \E_{\Psi}\exp\left(-\int_{\X}\int_S\min\left(\frac{1}{\tau},\max_{1\leq j\leq m}\frac{X_{\xi}(t_j-\xi)}{nz_j}\right)\Psi_i(\xi)~\d\xi\d\P_{X_\xi}\right)\\
	=&\ \E_{\Psi}\exp\left(-\int_{\X}\int_S\min\left(\frac{n}{\tau},\max_{1\leq j\leq m}\frac{X_{\xi}(t_j-\xi)}{z_j}\right)n^{-1}\sum_{i=1}^n\Psi_i(\xi)~\d\xi\d\P_{X_\xi}\right).
	\end{align*}
	Since
	$$\limm \min\left(\frac{n}{\tau},\max_{1\leq j\leq m}\frac{X_{\xi}(t_j-\xi)}{z_j}\right)n^{-1}\sum_{i=1}^n\Psi_i(\xi)=\max_{1\leq j\leq m}\frac{X_{\xi}(t_j-\xi)}{z_j},$$
	we obtain 
	$$\limm \P(\Pi^{(n)}(t_1)\leq z_1, \dots, \Pi^{(n)}(t_m)\leq z_m)=\exp\left[-\int_{\X}\int_S \max_{1\leq i\leq n}\left({z_j}^{-1}X_{\xi}(t_j-\xi)\right)~\d\xi~\d\P_X\right].$$
\end{proof}

\subsection{Process of visible storm centres}

We now consider the thinning function

$$	{f}^{\ast}_{\th}(\Phi_M;\xi,\mxi)=\prod_{\xxi'\in\Phi_M}\left[1-\1_{u_{\xi'}>u_{\xi}X_{\xi}(0)X_{\xi'}(\xi-\xi')^{-1}}\right].$$
The process resulting from dependent thinning equals 
\begin{align*}
{\Phi}^{\ast}_{\th}=&\ \{(\xi,\mxi)\in\Phi_M:{f}^{\ast}_{\th}(\Phi_M;\xi,\mxi)=1\}\\
=&\ \{(\xi,\mxi)\in\Phi_{\th}:\Pi(\xi)=\mxi(\xi)\}.
\end{align*}
That is, a point $(\xi,\mxi)$ is retained if $\mxi(\xi)\geq\mxi'(\xi)$ for all other $(\xi',\mxi')$ in $\Phi_M$. This condition is sharper than \eqref{eq:thinExHardCore} in the preceding section where points $(\xi,\mxi)$ are retained if there is an arbitrary $t$ such that $\mxi(t)\geq\mxi'(t)$ for all other $(\xi',\mxi')$ - therefore ${\Phi}^{\ast}_{\th}\subset\Phi_{\th}$.

\begin{figure}
	\centering
	\begin{minipage}[b]{6.5 cm}
		\includegraphics[width=1\textwidth]{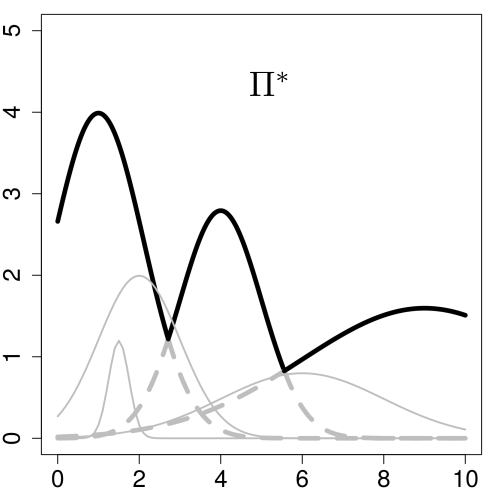}
	\end{minipage}
	\begin{minipage}[b]{6.5cm}
		\includegraphics[width=1\textwidth]{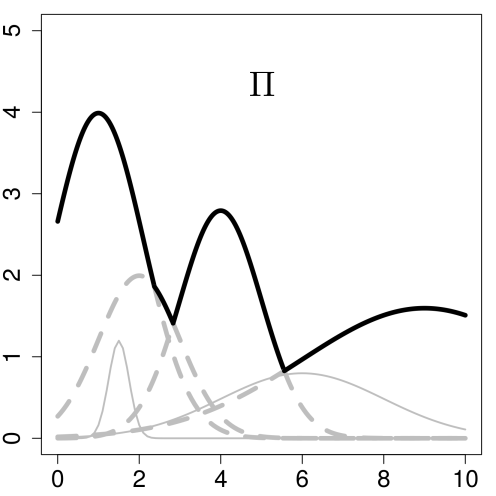}
	\end{minipage}
	\caption{The black solid lines form the final processes ${\Pi}^{\ast}$ (left) and $\Pi$ (right). The shape with centre $\xi=2$ does not contribute to ${\Pi}^{\ast}$, since its centre is covered by an other shape.}
	\label{fig:vis_storms}
\end{figure}

We introduce the \textit{process of visible storm centres} defined by
\begin{align}
\label{def:POVSC}
{\Pi}^{\ast}(t)=\bigvee_{(\xi,\mxi)\in{\Phi}^{\ast}_{\th}}\mxi(t).
\end{align}
This process is closely related to the extremal hard-core process $\Pi$ defined in \eqref{def:extremal_hardcore}, see also Figure~\ref{fig:vis_storms}.

\begin{Remark}
	The set ${\Phi}^{\ast}_{\th}$ is a subset of the set of extremal functions of $\Phi_M$ introduced by \cite{dombry2013contpoints}. If $\mxi$ is an extremal function which is not included in ${\Phi}^{\ast}_{\th}$, then ${\Pi}^{\ast}(\xi)>\mxi(\xi)$, i.e.\ the centre of $\mxi$ is covered by other storms, see Figure~\ref{fig:vis_storms}.
\end{Remark}

We apply Proposition~\ref{prop:FOI} and \ref{prop:soi_general} from Section~\ref{sec:coxmodel} to calculate first and second order properties of the thinned process $\Phi^{\ast}_{\th}$.
\begin{Lemma}
	Let $\Phi\sim\lgcp(\mu-\log\tau,C)$ with random intensity function $\tau^{-1}\Psi$.
	The intensity of ${\Phi}^{\ast}_{\th}$ is given by 
	
	$$\rho_{{\Phi}^{\ast}_{\th}}(\xi)=p_0\E\Psi(\xi)\int_{\X}X_{\xi}(o)\E_{\widetilde{\Psi}}\left[\frac{1-\exp(-\tau^{-1}X_{\xi}(0)^{-1}\cdot c_{\widetilde{\Psi}})}{c_{\widetilde{\Psi}}}\right]\d\P_{X_{\xi}},$$
	where
	$$c_{\widetilde{\Psi}}=\int_{S}\E_XX_{\xi'}(\xi-\xi')\widetilde{\Psi}(\xi')\d\xi'$$
	and $\tau^{-1}\widetilde{\Psi}$ is the random intensity function of $\widetilde{\Phi}\sim\lgcp(\widetilde{\mu}-\log\tau,C)$ and $\widetilde{\mu}(\cdot)=\mu(\cdot)+C(\cdot,\xi)$.
\end{Lemma}
\begin{proof}
	
	The retaining probability can be calculated by
	
	\begin{align*}
	r(\xi,\mxi)=&\ \E_{\widetilde{\Psi}}\exp\left(-\int_{S}\int_{\M}\zeta(\xxi,\xxi')p(\xxi,\xxi')\tau^{-1}\widetilde{\Psi}(\xi')\mu(dm_{\xi'})\d\xi'\right)\\
	=&\ \E_{\widetilde{\Psi}}\exp\left(-\int_{S}\int_{\M}\1_{u_{\xi'}>u_{\xi}X_{\xi}(0)X_{\xi'}(\xi-\xi')^{-1}}\tau^{-1}\widetilde{\Psi}(\xi')\mu(dm_{\xi'})\d\xi'\right)\\
	=&\ \E_{\widetilde{\Psi}}\exp\left(-\int_S\int_{\X}\int_{\tau}^{\infty}\1_{u_{\xi'}>u_{\xi}X_{\xi}(0)X_{\xi'}(\xi-\xi')^{-1}}\widetilde{\Psi}(\xi')u_{\xi'}^{-2}\d u_{\xi'}\d\P_{X_{\xi'}}\d\xi'\right).\\
	\end{align*}
	
	Due to the condition \eqref{eq:Xmonotone}, $u_{\xi}X_{\xi}(0)X_{\xi'}(\xi-\xi')^{-1}\geq\tau$ and therefore
	\begin{align*}
	r(\xi,\mxi)=&\ \E_{\widetilde{\Psi}}\exp\left(-\int_{\X}\int_S\int_{u_{\xi}X_{\xi}(0)X_{\xi'}(\xi-\xi')^{-1}}^{\infty}u_{\xi'}^{-2}\widetilde{\Psi}(\xi')\d u_{\xi'}\d\P_{X_{\xi'}}\d\xi'\right)\\	
	=&\ \E_{\widetilde{\Psi}}\exp\left(-\int_{\X}\int_{S} \frac{X_{\xi'}(\xi-\xi')}{u_{\xi}X_{\xi}(0)}\widetilde{\Psi}(\xi')\d\xi'\d\P_{X_{\xi'}}\right).
	\end{align*}
	The intensity then equals
	\begin{align*}
	\rho_{\Phi_{\th}}(\xi)=&\ p_0\rho_{\Phi}(\xi)\int_{\M}\E_{\widetilde{\Psi}}\exp\left(-\int_{\X}\int_{S} \frac{X_{\xi'}(\xi-\xi')}{u_{\xi}X_{\xi}(0)}\widetilde{\Psi}(\xi')\d\xi'\d\P_{X_{\xi'}}\right)\mu(\d m_{\xi})\\
	=&\ p_0\E\Psi(\xi)\int_{\X}\int_{\tau}^{\infty}\E_{\widetilde{\Psi}}\exp\left(-u^{-1}X_{\xi}(0)^{-1}c_{\widetilde{\Psi},X_{\xi}}\right)u^{-2}\d u\d\P_{X_{\xi}},\\
	\end{align*}
	with $c_{\widetilde{\Psi}}=\int_{\X}\int_{S}X_{\xi'}(\xi-\xi')\widetilde{\Psi}(\xi')\d\xi'\d\P_{X_{\xi'}}$. By calculating the integral with respect to $u$ we finally obtain
	
	\begin{align*}
	\rho_{\Phi_{\th}}(\xi)=&\ p_0\E\Psi(\xi)\int_{\X}X_{\xi}(0)\E_{\widetilde{\Psi}}\left[\frac{1-\exp(-\tau^{-1}X_{\xi}(0)^{-1}\cdot c_{\widetilde{\Psi}})}{c_{\widetilde{\Psi}}}\right]\d\P_{X_{\xi}}.
	\end{align*}
\end{proof}

\begin{Lemma}
	Let $\Phi\sim\lgcp(\mu-\log\tau^{-1},C)$ with random intensity function $\Psi$. The second order intensity of ${\Phi}^{\ast}_{\th}$ equals
	\begin{align*}
	\rho_{{\Phi}^{\ast}_{\th}}^{(2)}(\xi,\eta)=p_0^2\E[\Psi(\xi)\Psi(\eta)]\int_{\X}\int_{\X}\bigg[&\int_{\tau}^{\infty}\int_{\tau}^{\frac{u_{\eta}X_{\eta}(0)}{X_{\xi}(\eta-\xi)}}r(\xxi)r(\eeta)r(\xxi,\eeta)u_{\xi}^{-2}u_{\eta}^{-2}~\d u_{\xi}~\d u_{\eta}\\
	-&\int_{\tau}^{\infty}\int_{\frac{u_{\xi}X_{\xi}(0)}{X_{\eta}(\xi-\eta)}}^{\infty}r(\xxi)r(\eeta)r(\xxi,\eeta)u_{\xi}^{-2}u_{\eta}^{-2}~\d u_{\eta}~\d u_{\xi}\bigg]\d\P_{X_{\xi}}\d\P_{X_{\eta}}
	\end{align*}
	with $$r(\xxi)=\E_{\widetilde{\Psi}}\exp\left(-\int_{\X}\int_{S} \frac{X_{\xi'}(\xi-\xi')}{u_{\xi}X_{\xi}(0)}\widetilde{\Psi}(\xi')\d\xi'\d\P_{X_{\xi'}}\right)$$ and
	$$r(\xxi,\eeta)=\E_{\widetilde{\Psi}} \exp\bigg(\int_X\int_S \min\left(\frac{X_{\xi'}(\xi-\xi')}{u_{\xi}X_{\xi}(0)},\frac{X_{\xi'}(\eta-\xi')}{u_{\eta}X_{\eta}(0)}\right)\widetilde{\Psi}(\xi')\d\xi'\d\P_{X_{\xi'}}\bigg).$$
	Here  $\tau^{-1}\widetilde{\Psi}$ is the random intensity function of $\widetilde{\Phi}\sim\lgcp(\widetilde{\mu}-\log\tau^{-1},C)$ and $\widetilde{\mu}(\cdot)=\mu(\cdot)+C(\cdot,\xi)+C(\cdot,\eta)$.
	
\end{Lemma}
\begin{proof}
	
	Due to Proposition~\ref{prop:soi_general}
	\begin{align*}\rho_{\mathrm{th}}^{(2)}(\xi,\eta)=\rho^{(2)}_{\Phi}(\xi,\eta)p_0^2&\int_{\M}\int_{\M}h(\xxi,\eeta)h(\eeta,\xxi)\\
	&\E_{\widetilde{\Psi}} \exp\left(-\int_{S\times\M} \big(1-h(\xxi,\xxi')h(\eeta,\xxi')\big)\widetilde{\Psi}(\xi')\d\xi'\nu(\d\mxi')\right)\nu(\d\mxi)\nu(\d\meta).
	\end{align*}
	Then
	\begin{align*}
	&\ \E_{\widetilde{\Psi}} \exp\left(-\int_{S\times\M} \big(1-h(\xxi,\xxi')h(\eeta,\xxi')\big)\tau^{-1}\widetilde{\Psi}(\xi')\d\xi'\nu(\d\mxi')\right)\\
	=&\ \E_{\widetilde{\Psi}} \exp\bigg(-\int_X\int_S\int_{\tau}^{\infty} [\1_{u>u_{\xi}X_{\xi}(0)(X_{\xi'}(\xi-\xi'))^{-1}}+\1_{u>u_{\eta}X_{\eta}(0)(X_{\xi'}(\eta-\xi'))^{-1}}\\
	&-\1_{u>u_{\xi}X_{\xi}(0)(X_{\xi'}(\xi-\xi'))^{-1}}\1_{u>u_{\eta}X_{\eta}(0)(X_{\xi'}(\eta-\xi'))^{-1}}]\widetilde{\Psi}(\xi')u^{-2}~\d u\d\xi'\d\P_{X_{\xi'}}\bigg)\\
	=&\ r(\xxi)r(\eeta)\E_{\widetilde{\Psi}} \exp\bigg(-\int_X\int_S\int_{\tau}^{\infty}\1_{u>\max[u_{\xi}X_{\xi}(0)(X_{\xi'}(\xi-\xi'))^{-1},u_{\eta}X_{\eta}(0)(X_{\xi'}(\eta-\xi'))^{-1}]}\widetilde{\Psi}(\xi')u^{-2}~\d u\d\xi'\d\P_{X_{\xi'}}\bigg)\\		
	=&\ r(\xxi)r(\eeta)\E_{\widetilde{\Psi}} \exp\bigg(\int_X\int_S \min\left(\frac{X_{\xi'}(\xi-\xi')}{u_{\xi}X_{\xi}(0)},\frac{X_{\xi'}(\eta-\xi')}{u_{\eta}X_{\eta}(0)}\right)\widetilde{\Psi}(\xi')\d\xi'\d\P_{X_{\xi'}}\bigg)
	\end{align*}
	with $r(\xxi)=\E_{\widetilde{\Psi}}\exp\left(-\int_{\X}\int_{S} \frac{X_{\xi'}(\xi-\xi')}{u_{\xi}X_{\xi}(0)}\widetilde{\Psi}(\xi')\d\xi'\d\P_{X_{\xi'}}\right)$.  
	Furthermore
	\begin{align*}
	h(\xxi,\eeta)h(\eeta,\xxi)=1&-\1_{u_{\eta}>u_{\xi}X_{\xi}(0)X_{\eta}(\xi-\eta)^{-1}}-\1_{u_{\xi}>u_{\eta}X_{\eta}(0)X_{\xi}(\eta-\xi)^{-1}}\\
	&+\1_{u_{\eta}>u_{\xi}X_{\xi}(0)X_{\eta}(\xi-\eta)^{-1}}\1_{u_{\xi}>u_{\eta}X_{\eta}(0)X_{\xi}(\eta-\xi)^{-1}}\\
	=1&-\1_{u_{\eta}>u_{\xi}X_{\xi}(0)X_{\eta}(\xi-\eta)^{-1}}-\1_{u_{\xi}>u_{\eta}X_{\eta}(0)X_{\xi}(\eta-\xi)^{-1}}
	\end{align*}
	and 
	\begin{align*}
	\rho_{{\Phi}^{\ast}_{\th}}^{(2)}(\xi,\eta)=p_0^2\E[\Psi(\xi)\Psi(\eta)]\bigg[\int_{\X}\int_{\tau}^{\infty}&\int_{\X}\int_{\tau}^{\infty}r(\xxi)r(\eeta)r(\xxi,\eeta)u_{\xi}^{-2}u_{\eta}^{-2}~\d u_{\xi}~\d u_{\eta}\d\P_{X_{\xi}}\d\P_{X_{\eta}}\\
	-\int_{\X}\int_{\tau}^{\infty}&\int_{\X}\int_{\frac{u_{\eta}X_{\eta}(0)}{X_{\xi}(\eta-\xi)}}^{\infty}r(\xxi)r(\eeta)r(\xxi,\eeta)u_{\xi}^{-2}u_{\eta}^{-2}~\d u_{\xi}\d\P_{X_{\xi}}~\d u_{\eta}\d\P_{X_{\eta}}\\
	-\int_{\X}\int_{\tau}^{\infty}&\int_{\X}\int_{\frac{u_{\xi}X_{\xi}(0)}{X_{\eta}(\xi-\eta)}}^{\infty}r(\xxi)r(\eeta)r(\xxi,\eeta)u_{\xi}^{-2}u_{\eta}^{-2}~\d u_{\eta}\d\P_{X_{\eta}}~\d u_{\xi}\d\P_{X_{\xi}}\bigg]\\
	=p_0^2\E[\Psi(\xi)\Psi(\eta)]\int_{\X}\int_{\X}\bigg[&\int_{\tau}^{\infty}\int_{\tau}^{\frac{u_{\eta}X_{\eta}(0)}{X_{\xi}(\eta-\xi)}}r(\xxi)r(\eeta)r(\xxi,\eeta)u_{\xi}^{-2}u_{\eta}^{-2}~\d u_{\xi}~\d u_{\eta}\\
	-&\int_{\tau}^{\infty}\int_{\frac{u_{\xi}X_{\xi}(0)}{X_{\eta}(\xi-\eta)}}^{\infty}r(\xxi)r(\eeta)r(\xxi,\eeta)u_{\xi}^{-2}u_{\eta}^{-2}~\d u_{\eta}~\d u_{\xi}\bigg]\d\P_{X_{\xi}}\d\P_{X_{\eta}}.
	\end{align*}
\end{proof}

\begin{Remark}
	Let $Z$ be the classical Smith model \citep{smith1990maxstable} in $\R^2$, that is $X$ is the density of the two-dimensional standard-normal distribution. Then, the intensity of the process of visible storm centres of $Z$ can be calculated as
	$$\lim_{\tau\rightarrow 0}\rho_{{\Phi}^{\ast}_{\th}}(\xi)=X_{\xi}(o)\left[{\int_{S}X_{\xi'}(\xi-\xi')\d\xi'}\right]^{-1}=(2\pi)^{-1}.$$
	In general, $\rho_{{\Phi}^{\ast}_{\th}}$ and $\rho_{{\Phi}^{\ast}_{\th}}^{(2)}$ cannot be explicitly calculated if $\Psi$ is random. However, numerical calculation of $\rho_{{\Phi}^{\ast}_{\th}}$ is feasible in most cases. For certain choices of $X$, e.g.\ $X(t)=(1-t^2)_+$ in $\R$, even $\rho_{{\Phi}^{\ast}_{\th}}^{(2)}$ is numerically tractable.
\end{Remark}

\vspace{5mm}


{\small
	\textbf{Acknowledgments.}
The research of the first author was partly supported by the DFG through 'RTG 1953 - Statistical Modeling of Complex Systems and Processes' and Volkswagen Stiftung within the project 'Mesoscale Weather Extremes - Theory, Spatial Modeling and Prediction'. The authors are grateful to M.~Kroll and K.~Strokorb for valuable comments and suggestions.
}
\bibliography{Literatur}
\bibliographystyle{plainnat}
\end{document}